\documentclass{amsart}
\usepackage{amsmath,hyperref,amsfonts,amssymb,amsthm,mathrsfs,mathtools,comment}
\usepackage{pdfpages} 
\usepackage{inputenc}
\usepackage{graphicx}
\usepackage[T1]{fontenc}
\usepackage{lmodern}

\theoremstyle{plain}
\newtheorem{theorem}{Theorem}[section]
\newtheorem{prop}[theorem]{Proposition}
\newtheorem{proposition}[theorem]{Proposition}
\newtheorem{lemma}[theorem]{Lemma}
\newtheorem{corollary}[theorem]{Corollary}
\newtheorem{conjecture}[theorem]{Conjecture}
\newtheorem{definition}[theorem]{Definition}

\theoremstyle{remark}
\newtheorem{remark}[theorem]{Remark}
\newtheorem*{remark*}{Remark}

\theoremstyle{remark}
\newtheorem{notation}[theorem]{Notation}
\newtheorem*{notation*}{Notation}

\newcommand{\fp}{\mathfrak{p}}

\newcommand{\cS}{\mathcal{S}}
\newcommand{\cK}{\mathcal{K}}

\newcommand{\cA}{\mathcal{A}}
\newcommand{\cO}{\mathcal{O}}
\newcommand{\cH}{\mathcal{H}}
\newcommand{\cT}{\mathcal{T}}
\newcommand{\cC}{\mathcal{C}}
\newcommand{\cM}{\mathcal{M}}

\DeclareMathOperator{\sgn}{sgn}

\newcommand{\Spec}{\mathrm{Spec}}
\newcommand{\Sp}{\mathrm{Sp}}
\newcommand{\disc}{\mathrm{disc}}
\newcommand{\GL}{\mathrm{GL}}

\renewcommand{\bar}{\overline}

\newcommand{\pfloor}[1]{\left\lfloor #1 \right\rfloor}

\newcommand{\1}{\mathbf{1}}
\newcommand{\A}{\mathbb{A}}

\newcommand{\F}{\mathbb{F}}

\newcommand{\Q}{\mathbb{Q}}
\newcommand{\Z}{\mathbb{Z}}
\newcommand{\R}{\mathbb{R}}

\newcommand{\Rsd}{R_s^{\sharp}}

\begin{document}

\title[Generically ordinary hyperelliptic families]{Construction of generically ordinary families of hyperelliptic curves}

\author{Hui June Zhu}

\address{
Department of Mathematics,
University at Buffalo, 
State University of New York,
Buffalo, NY 14260, USA}
\email{hjzhu@math.buffalo.edu}

\date{Last updated August 23rd, 2026}

\begin{abstract}
Katz conjectured in a 2018 lecture 
that the family of curves $y^2=x^d-dx+t$ 
over the $t$-line 
is generically ordinary for all sufficiently large primes $p$. 
We prove that, for every $g\ge 2$ and every nonzero algebraic 
integer $\alpha$, 
the  genus-$g$ families 
$\cC_\alpha: y^2=x^d+\alpha x+t$ where $d\in\{2g+1, 2g+2\}$ 
are generically ordinary at every prime 
$p>P^+(d)$, provided that $\alpha$ is nonzero modulo every prime above $p$.
The bound $P^+(d)=d^2-4d+2$ if $d$ is odd, 
and $P^+(d)=(d^2-3d+2)/2$ if $d$ is even. 
Finally, we show that for every algebraic integer $\alpha$ and 
every prime $p\equiv 1\pmod d$, the family $\cC_\alpha$ is generically ordinary at every prime $\fp$ above $p$. 
\end{abstract}

\maketitle

\section{Introduction and statement of results}

Let $\cC$ be a smooth projective curve of genus $g$ over a field of
characteristic $p>0$.  
The curve 
$\cC$ is {\it ordinary} if its Jacobian is an ordinary 
abelian variety, that is, its  Newton polygon has
only slopes $0$ and $1$.
For an elliptic curve over a finite fields,
the classical {\em Hasse invariant} determines its ordinarity. 
In higher genus, the analogous object is the Hasse--Witt matrix, 
or its dual the
Cartier--Manin matrix, 
and ordinarity is detected by the nonvanishing of
its determinant. 

The ordinary locus in the moduli space of curves $\cM_g$ is
open, dense, geometrically irreducible, and the corresponding homomorphism 
$\pi_1\rightarrow 
\GL_g(\F_p)$ is surjective, see \cite{FG04,FC90}. 
Analogous statements hold for the ordinary locus in the hyperelliptic locus $\cH_g$ 
(see \cite{GP05,AP08}, and also \cite{PZ12} for the characteristic-$2$ case).

Given a one-parameter family of hyperelliptic curves over a number field,
with ``large monodromy''--- that is, whose topological monodromy is of finite index in $\Sp_{2g}(\Z)$---  does it always meet the ordinary locus in $\cH_g$ for every prime $p$? 
Katz's family $$\cK: y^2=x^d-dx+t,$$ 
over (the smooth locus of) the $t$-line has large monodromy 
(as shown in \cite{Kat14}). However, 
$\cK$ may fail to be generically ordinary for small $p$ 
(see \cite{Kat18}).
   
\begin{conjecture}[Katz 2018]
\label{Con:Katz}
Let $d\in\{2g+1,2g+2\},$ and $g\ge 2$.
There exists a quadratic polynomial $P^+(d)$ 
such that the family $\cK$ is generically ordinary 
at $p$ for all $p>P^+(d)$. 
\end{conjecture}

\begin{remark*}[The elliptic curve case]
In the case $g=1$ (that is $d=3,4$) the short Weierstrass family 
$y^2=x^d+\alpha x+t$ over the $t$-line 
is generically ordinary at every prime $p>3$ and $p\nmid \alpha$.
The case $g=2$ can be proved by Will Sawin's argument 
 \cite{Saw16}. 
This paper settles this conjecture for all $g\ge 2$.
\end{remark*}

We use the following standing notation throughout the paper:

\begin{notation}\label{N:main}
Let $d\in\{2g+1,2g+2\}$ and $g\ge 2$. 
Let $\alpha$ be an algebraic integer 
and let 
\begin{equation*}
        \cC_\alpha(t):\quad y^2=x^d+\alpha x+t.
\end{equation*}
Let $D_\alpha(t)$ denote 
the discriminant $\disc_x(x^d+\alpha x+t)$.
Denote 
$$
\cS^\circ\coloneqq \Spec \cO_{\Q(\alpha)}[1/2,t,D_\alpha(t)^{-1}].
$$
Let   
$$
P^+(d)=
\begin{cases}
d^2-4d+2 & \mbox{ if }d=2g+1\\
(d^2-3d+2)/2 & \mbox{ if }d=2g+2.
\end{cases}
$$
\end{notation}

\begin{theorem}\label{T:main}
Let $\alpha$ be a nonzero algebraic integer and let $F=\Q(\alpha)$.
Then for every prime $p>P^+(d)$ and $p\nmid N_{F/\Q}(\alpha)$, 
the family $\cC_\alpha$ of genus-$g$ curves 
$$
\cC_\alpha \longrightarrow  \cS^\circ
$$
is generically ordinary at each prime $\fp\mid p$ of $F$. 
\end{theorem}

\begin{remark}
Notice that the Katz family $\cK=\cC_{-d}$ is a special case;
hence Conjecture \ref{Con:Katz} follows from Theorem \ref{T:main}.
\end{remark}

\begin{remark}
The bound $P^+(d)$ cannot be weakened to $p\ge P^+(d)$ in the theorem.
The data provided by Katz in \cite{Kat18}
shows that the generic ordinarity fails for $p=P^+(d)$ 
in the following cases:
$$(d,p)=(21,359), (31,839), (51,2399),$$
and further computations suggest the same for 
$$(65,3967),(105, 10607).$$
In all of these cases, $d$ is odd and $p\equiv 2\bmod d$. 
\end{remark}

For prime $p\equiv 1\pmod d$ we have
a refinement in the following Corollary \ref{C:bound}.

\begin{corollary}\label{C:bound}
Let $\alpha$ be any algebraic integer and let $F=\Q(\alpha)$.
Then for every prime $p\equiv 1\bmod d$ 
the family $\cC_\alpha$ of genus-$g$ curves 
$$\cC_\alpha \longrightarrow \cS^\circ$$
is generically ordinary at each prime $\fp\mid p$ of $F$.
\end{corollary}

\begin{remark}
Katz showed in \cite{Kat14} 
that for any Morse polynomial $h(x)$ of degree $d$ 
the family of curves $y^2=h(x)+t$ has large monodromy. 
It is straightforward to verify by definition 
that $h(x)=x^d+\alpha x$ 
in characteristic $0$ is Morse if and only if $\alpha$ is nonzero.
In characteristic $p$ it is Morse if and only if  $p\nmid d(d-1)$ and  
$\alpha$ is nonzero. 
This shows that $\cC_\alpha$ has topological monodromy of finite index in $\Sp_{2g}(\Z)$. Hence Theorem \ref{T:main} shows that
part of Conjecture 1 of \cite{Kat18} holds for this family 
$\cC_\alpha$.
\end{remark}

For $g\ge2$ we have $P^+(d)> d$. Every 
prime $p>P^+(d)>d$ implies $p\nmid d(d-1)$.
At every prime $\fp\mid p$ for which $\alpha\not\in \fp$, 
the polynomial 
$x^d+\alpha x\bmod \fp$ is Morse. 
This implies, by Katz's argument \cite{Kat14}, that 
$\cC_\alpha$ at $\fp$ has large monodromy.

\vspace{5mm}

Our result lies in the tradition of Miller's paper \cite{Mil72}. 
Miller exhibited two families over the $t$-line 
$$
\begin{cases}
        y^2&=x^{2g+1}+t x^{g+1}+x \quad\mbox{ if $p\nmid g$}\\
y^2&=x^{2g+2}+t x^{g+1}+1 \quad \mbox{ if $p\mid g$}
\end{cases}
$$
are generically ordinary. 
The first one was in fact proposed by 
J. Lubin and N. Katz (unpublished). 
For Miller's families, the coefficient matrix has the support of a permutation
matrix. By contrast, the matrix for $\cK$
is generally not monomial and 
can contain many nonzero entries, so its determinant need not reduce to a single product. 
A different and new approach is needed. 

\vspace{4mm}
This paper is organized as follows. 
We present in Section 2 how to evaluate a determinant 
$\Delta$ and prove that prime factors of $\Delta$ are bounded above by  $P^+(d)$.
In Section 3, we define the Hasse--Witt polynomial $H_{p,\alpha}$, 
of $\cC_\alpha$ over $\cO_{\Q(\alpha)}$, 
and show that it is nonzero at each prime $\fp\mid p$ 
for $p>P^+(d)$ and $p\nmid N_{\Q(\alpha)/\Q}(\alpha)$.
This is accomplished by showing 
its leading coefficient is, up to a $p$-adic unit factor, 
congruent to $\Delta$ mod $p$. 
In Section 4 we prove the generic ordinarity
of the family $\cC_\alpha(t)$. Theorem \ref{T:main2}
proves  Theorem~\ref{T:main}. 
We then deduce Katz's Conjecture \ref{Con:Katz} in 
Corollary \ref{C:integer}.
As a natural application we also derive a similar result for a two-parameter family over the affine plane
in Corollary \ref{C:plane}. 
Finally we prove Corollary~\ref{C:bound}.

\section{Prime factors of a determinant}

\subsection{Product formula for a binomial determinant}
\label{S:Product}

The goal of this subsection is to produce a product formula for a determinant $\Delta$ in Proposition \ref{P:delta}.
The core technique in the proof is to reduce it to a 
Vandermonde-type determinant, which has a natural product formula.

For any integer $1\le r\le g-1$ we define a positive 
rational number
\begin{equation}\label{E:lambda}
\lambda_r=\frac{d(2r+1)}{2(d-1)}.
\end{equation}

\begin{lemma} \label{L:1}
For $1\le r\le g-1$, $\lambda_r$ is not an integer. 
\end{lemma}
\begin{proof}
First notice $3\le 2r+1<d-1$.
If $d$ is odd, then $(2(d-1),d)=1$. Since $2(d-1)\nmid (2r+1)$, $\lambda_r\not\in\Z$.
If $d$ is even, then $\lambda_r=\frac{\frac{d}{2}(2r+1)}{d-1}$.
Since $(d-1,\frac{d}{2})=1$ and $d-1\nmid 2r+1$, $\lambda_r\not\in\Z$.
\end{proof}

\begin{lemma}\label{L:product}
Let $
P_j(T)
=
\prod_{r=1}^{j-1}(T-r)
\prod_{r=j}^{g-1}(T-\lambda_r).$
Let
$$
P(x_1,\dots,x_g)=
\det_{1\le i,j\le g} P_j(x_i).
$$
Then 
\[
P(x_1,\dots,x_g)=
\displaystyle \prod_{i<j}(x_j-x_i)
\displaystyle \prod_{r=1}^{g-1}\prod_{k=1}^r
(k-\lambda_r).\]
\end{lemma}

\begin{proof} 
The determinant $P$ is alternating in $x_1,\ldots,x_g$,
hence $\prod_{i<j}(x_j-x_i)$ divides $P$.
Since $x_i$ occurs only in the $i$-th row and each $P_j$ has degree $g-1$, we have 
$\deg_{x_i}P\le g-1$ for each $i$.
Meanwhile, $\prod_{i<j}(x_j-x_i)$ has degree exactly $g-1$ in each $x_i$, 
hence 
\begin{equation}\label{E:H}
P(x_1,\dots,x_g)
=
C\prod_{i<j}(x_j-x_i)
\end{equation}
for some constant $C$.
To compute $C$, set $x_i=i$ for $1\le i\le g$. 
The matrix $(P_j(i))_{i,j}$ is lower triangular with diagonal entries
\[
P_i(i)
=
\prod_{r=1}^{i-1}(i-r)
\prod_{r=i}^{g-1}(i-\lambda_r)
=
(i-1)!\prod_{r=i}^{g-1}(i-\lambda_r).
\]
Therefore, \eqref{E:H} becomes
\begin{align*}
P(1,2,\dots,g)
&=\prod_{i=1}^{g}P_i(i)=
\prod_{i=1}^g\left(
(i-1)!\prod_{r=i}^{g-1}(i-\lambda_r)\right)
=
\left(\prod_{i=1}^g(i-1)!\right)
\prod_{i=1}^{g}\prod_{r=i}^{g-1}(i-\lambda_r)
\\
&=
C\prod_{i<j}(j-i)=C\prod_{i=1}^g(i-1)!.
\end{align*}
Thus we get
$
C
=
\prod_{i=1}^g\prod_{r=i}^{g-1}(i-\lambda_r)
=
\prod_{r=1}^{g-1}\prod_{k=1}^r(k-\lambda_r).
$
Substituting $C$ back to \eqref{E:H}, this finishes the proof.
\end{proof}

We define for any nonnegative integer $k$,
$$
\binom{z}{k}=\frac{z(z-1)\cdots (z-k+1)}{k!}
$$
and set $\binom{z}{k}=0$ for $k<0$. 
Moreover, if $z\in\Z_{\ge 0}$ and $z<k$ then $\binom{z}{k}=0$.
For the rest of the paper 
fix a set
$$\cA \coloneqq\{1,2,\ldots,d-1\}\backslash \{\frac{d}{2}\}.$$
When $d$ is odd, the exclusion is vacuous.
Let $s_1,\ldots,s_g\in\cA$ be distinct integers.
Let 
\begin{equation}\label{E:alpha}
\beta_i\coloneqq\frac{s_i}{d}-\frac12.
\end{equation}

\begin{remark}\label{R:binom}
Note that $\frac{-1}{2}<\beta_i<\frac12$ and $\beta_i\ne 0$.
This implies that $\binom{\beta_i}{s_i-1}\ne 0$.
\end{remark}

The determinant $\Delta$ below 
will later occur as the reduction mod $p$
of a deciding factor $\Delta_p$
in the leading coefficient of the Hasse--Witt polynomial.
(See Proposition \ref{P:D_p}.)

\begin{proposition}[Product formula]
\label{P:delta}
Define
\begin{equation*}\label{E:Delta}
\Delta(s_1,\ldots,s_g)
=
\det_{1\le i,j\le g}
\binom{\beta_i}{s_i-j}.
\end{equation*}
Then we have  
\[\boxed{
\Delta(s_1,\ldots,s_g)
=
\left(\frac{-d}{d-1}\right)^G
\displaystyle \prod_{i<j}(s_j-s_i)\cdot
\displaystyle 
\left(\prod_{r=1}^{g-1}
\prod_{k=1}^r(k-\lambda_r)\right)
\cdot
\left(\prod_{i=1}^g\frac{ \binom{\beta_i}{s_i-1}}
 { \prod_{r=1}^{g-1}(s_i-\lambda_r)}\right)
}
\]
where $G=g(g-1)/2$.
Moreover, $\Delta$ is a nonzero rational number.
\end{proposition}

\begin{proof}
We first reduce this determinant to $P(s_1,\ldots,s_g)$ from above.
Notice from Lemma \ref{L:1} that $s_i-\lambda_r\ne 0$ since $s_i$ is integer and 
$\lambda_r$ is not, so the denominators on the right-hand side are nonzero.
Thus the right-hand side of the formula is well-defined.

Since $\binom{\beta_i}{s_i-1}\ne 0$ by Remark \ref{R:binom},
we rewrite 
\begin{equation}\label{E:Delta_R}
\Delta(s_1,\ldots,s_g)
=
\left(\prod_{i=1}^g \binom{\beta_i}{s_i-1}\right)
\det_{1\le i,j\le g} K_{ij},
\end{equation}
where 
$
K_{ij}
\coloneqq
\frac{\binom{\beta_i}{s_i-j}}{\binom{\beta_i}{s_i-1}}.
$
For $j=1$, $K_{i1}=1$; for $j\ge 2$ 
we have 
$$
K_{ij}
=
\prod_{r=1}^{j-1}
\frac{s_i-r}{\beta_i-s_i+r+1}=\left(\frac{-d}{d-1}\right)^{j-1}\frac{1}{\prod_{r=1}^{g-1}(s_i-\lambda_r)}
\cdot P_j(s_i).
$$
Thus by Lemma \ref{L:product},
\begin{align*}
\det_{1\le i,j\le g}(K_{ij})
&=
\left(\frac{-d}{d-1}\right)^G \frac{1}{\prod_{i=1}^{g}
\prod_{r=1}^{g-1}(s_i-\lambda_r)}
P(s_1,\dots,s_g)\\
&=
\left(\frac{-d}{d-1}\right)^G \cdot
\frac{
\prod_{i<j}(s_j-s_i)
 \prod_{r=1}^{g-1}\prod_{k=1}^r(k-\lambda_r)
}{
\prod_{i=1}^g\prod_{r=1}^{g-1}(s_i-\lambda_r)
}.
\end{align*}
Substituting this into \eqref{E:Delta_R} proves our formula. 

We see that $\Delta$ is a nonzero rational number 
since the $s_i$ are distinct, and $\lambda_r$ 
is not integer by Lemma \ref{L:1}, so  $s_i-s_j\ne 0$ and $k-\lambda_r\ne 0$.
As we have seen above, $\binom{\beta_i}{s_i-1}\ne 0$,
hence all factors in the numerator are nonzero. 
\end{proof}

\subsection{A quadratic bound for the prime factors}

To obtain the desired bound $P^+(d)$ it is not sufficient to 
simply estimate each factor in the product formula of $\Delta$ in Proposition 
\ref{P:delta}. 
It is necessary to take certain cancellations into 
account (see Lemma \ref{L:prime}). 

Consider two regions in $\Z^2$:
\begin{align*}
U &=\{(x,y)\in\Z^2 \mid  g+1\le x\le d-1, g\le y\le x-1, x\in \cA\}\\
V &=\{(x,y)\in\Z^2 \mid  1\le x < g, x\le y\le g-1, x\in \cA\}.
\end{align*}

\begin{lemma}\label{L:h}
Let $h(x,y)=(2d-2)x-2dy-d$.
Then $h$ is an integral-valued function on $U$ and $V$
such that
$$
\max_{U}|h(x,y)|=
\begin{cases}
d^2-4d+2 & \mbox{ if $d=2g+1$}\\
d^2-3d+2 &  \mbox{ if $d=2g+2$}. 
\end{cases}
$$
$$
\max_{V}|h(x,y)|=
\begin{cases}
d^2-4d+2 &\mbox{ if $d=2g+1$}\\
d^2-5d+2 &\mbox{ if $d=2g+2$}.
\end{cases}
$$
\end{lemma}

\begin{proof}
(1) We opt to give a detailed and direct proof for 
the first case about $U$.
Since $h$ is a linear function in both variables, its extreme values are achieved
at vertices of the convex hull of its (lattice) domain $U$. 
We separate into two cases:\\
\noindent $\bullet$ 
If $d=2g+1$ then the vertex set of the convex hull of 
$U$ is $\{(g+1,g),(2g,g), (2g,2g-1)\}$.
Let 
$h_{\min}$ and $h_{\max}$ denote the minimal and maximal values at 
these vertices, respectively. Then
\begin{align*}
h_{\min}&=\min(h(g+1,g),h(2g,g),h(2g,2g-1))=2-d,\\
h_{\max}&=\max(h(g+1,g),h(2g,g),h(2g,2g-1))= d^2-4d+2.
\end{align*}
It follows that $\max_U |h(x,y)|=\max(h_{\max},-h_{\min})=d^2-4d+2.$\\
\noindent $\bullet$ If $d=2g+2$ then $x\ne g+1$, so the vertex set of $U$ 
is $\{(g+2,g),(2g+1,g),(2g+1,2g), (g+2,g+1)\}$. Then
$h_{\min}=2-d$ and $h_{\max}=d^2-3d+2$. It follows that $\max_U|h(x,y)|=d^2-3d+2$.

\vspace{2mm}

(2) The vertex set of the convex hull of $V$ is 
$\{(1,1),(1,g-1), (g-1,g-1)\}$ for all $d$.
For $d=2g+1$, $h_{\min}=h(1,g-1)=-(d^2-4d+2)$ 
and $h_{\max}=h(1,1)=-d-2$, hence $\max_V|h(x,y)|=d^2-4d+2$.
For $d=2g+2$, we have
$h_{\min}=h(1,g-1)=-(d^2-5d+2)$ and 
$h_{\max}=h(1,1)=-d-2$, thus in this case 
$\max_V|h(x,y)|=d^2-5d+2$. 
\end{proof}

Let
\begin{eqnarray*}
J_r&=&\prod_{k=1}^r(k-\lambda_r) \qquad \mbox{for $1\le r\le g-1$,}\\
R_s &=&\frac{\binom{\frac{s}{d}-\frac{1}{2}}{s-1}}{\prod_{r=1}^{g-1}(s-\lambda_r)}\qquad \mbox{ for $s\in \cA$}.
\end{eqnarray*}
For any $\zeta\in\Q^*$, by a {\it prime factor} of $\zeta$, 
we mean a prime $p$ such that $v_p(\zeta)\ne 0$.

\begin{lemma}\label{L:prime}
Every prime factor of
$J_r$ or $R_s$ 
is at most $P^+(d)$.
\end{lemma}

\begin{proof}
(1) We first consider $J_r$.
Notice
$$
k-\lambda_r=\frac{2(d-1)k-(2r+1)d}{2(d-1)}=\frac{h(k,r)}{2(d-1)}
$$
where $h(k,r)$ ranges over 
$W:=\{(k,r)\in\Z^2| 1\le k\le r\le g-1\}$, whose convex hull has  
vertices $\{(1,1), (1,g-1), (g-1,g-1)\}$.
When $d=2g+1$,  $|h(k,r)|\le d^2-4d+2$;  
when $d=2g+2$, each $h(k,r)$ is even, so each prime factor is bounded by 
 $\max\{2,|h(k,r)/2|\}\le (d^2-5d+2)/2$.
On the other hand, prime factors in the denominator of $k-\lambda_r$ are bounded by 
$d-1\le P^+(d)$.
This proves the statement for $J_r$.

(2) We now consider $R_s$.
Write its numerator and denominator 
as  quotients of products of integers:
\begin{align*}
\prod_{r=1}^{g-1}(s-\lambda_r)
&=\frac{\prod_{r=1}^{g-1} ((2s-2r-1)d-2s)}{(2(d-1))^{g-1}}\\
\binom{\frac{s}{d}-\frac{1}{2}}{s-1}&=\frac{\frac{2s-d}{2d}\cdot
\frac{2s-d-2d}{2d}\cdot \frac{2s-d-4d}{2d} \cdot
\cdots \frac{2s-d-2(s-2)d}{2d}}{(s-1)!}\nonumber\\
&=\frac{\prod_{\ell=1}^{s-1}((2\ell-1)d-2s)}{(-1)^{s-1}(s-1)!(2d)^{s-1}}.
\end{align*}
Taking their quotient, we have the following factorization:
$$
R_s = \frac{2^{g-s}(d-1)^{g-1}}{(-1)^{s-1}(s-1)!d^{s-1}}\Rsd
$$
where
\begin{equation*}
\Rsd:=\frac{\prod_{\ell=1}^{s-1}((2\ell-1)d-2s)}{\prod_{r=1}^{g-1} ((2s-2r-1)d-2s)}
\stackrel{(\star)}{=}
\frac{\prod_{\ell=1}^{s-1}h(s,\ell)}{\prod_{\ell=1}^{g-1}h(s,\ell)}.
\end{equation*}
The last equality $(\star)$ is by
reindexing the numerator via $\ell\mapsto s-\ell$, 
as it permutes $\{1,\ldots,s-1\}$.
The first factor of $R_s$ contains no prime factor $>d$,
so it remains to consider $\Rsd$.
The function $h(s,\ell)$ is integer-valued, 
its domain is either $U$ or $V$ depending on where $s$ lies:\\
\noindent 
(case 1): Suppose $g\le s\le d-1$. 
Then 
$\Rsd=\prod_{\ell=g}^{s-1}h(s,\ell).
$
For $s=g$, $\Rsd=1$, there are no remaining factors. 
It remains to consider $g+1\le s\le d-1$. Hence the function $h(s,\ell)$ ranges in the domain $U$. By Lemma \ref{L:h},
for $d=2g+1$, we have $|h(s,\ell)|\le d^2-4d+2$;
for $d=2g+2$, the integer $h(s,\ell)$ is even on the domain $U$, 
so every odd prime factor already divides $\frac{h(s,\ell)}{2}$.
Hence each prime factor is bounded by $\max\{2,|\frac{h(s,\ell)}{2}|\}
\le (d^2-3d+2)/2$ 
by the same lemma.
  
\noindent
(case 2): Suppose $1\le s<g$. 
Then 
$
\Rsd =\prod_{\ell=s}^{g-1}h(s,\ell)^{-1}.
$ 
In this case the function $h(s,\ell)$ has its domain $V$.
For $d=2g+1$, we have $|h(s,\ell)|\le d^2-4d+2$;
For $d=2g+2$, the integer $h(s,\ell)$ is even on its domain,
hence each odd prime factor divides $h(s,\ell)/2$ and hence
every prime factor is at most 
$\max\{2,|h(s,\ell)/2|\}\le (d^2-5d+2)/2$, by Lemma \ref{L:h}.

In summary, we have shown that for $d=2g+1$, any prime factor 
$p\le \max(d^2-4d+2,d)=P^+(d)$; for $d=2g+2$, 
any prime factor $p\le \max((d^2-3d+2)/2,d)=P^+(d)$.
\end{proof}

\begin{prop}\label{P:determinant} 
Let $s_1,\ldots,s_g\in \cA$ be distinct integers.
Each prime factor of $\Delta$ is at most $P^+(d)$.
If $p>P^+(d)$, then $\Delta\in\Z^*_{(p)}$ and 
$\Delta \not\equiv 0\pmod p$. 
\end{prop}

\begin{proof}
Since $1\le s_i,s_j\le d-1$, we have $|s_j-s_i|\le d-2<d.$ 
The factors $s_i-s_j$ and the prefactor $(-d/(d-1))^G$ contribute only prime factors at most $d$.
The remaining factor is
$\prod_{r=1}^{g-1}J_r\prod_{i=1}^gR_{s_i}$, whose prime 
factors
are at most $P^+(d)$ by Lemma \ref{L:prime}.
This proves the first statement.

For any $p>P^+(d)$, our first statement
implies that both the numerator and denominator of $\Delta$
are $p$-adic units.
Thus $\Delta\in\Z_{(p)}^*$ and  $\Delta\not\equiv 0\bmod p$.
\end{proof}

\section{Hasse--Witt polynomial over $\Z$}

For the rest of the paper let $p$ be a prime coprime to $2d$. 
In this section, we define Hasse--Witt polynomial with integral coefficients in one or two variables, then show that they are nonzero 
at primes lying over $p$ when $p$ is large enough. 

We use 
the following notation for the rest of the paper:
$$Q_i=\pfloor{\frac{p\,i}{d}}, \quad s_i=[p\,i]_d, \quad \mbox{  for $1\le i\le g$}.$$
Here $[N]_d$ represents the least nonnegative residue $N\bmod d$.
Once $p$ is given, $s_i$ can be considered a function of $i$.

\subsection{Leading term in the coefficient polynomial}

\begin{lemma}\label{L:S_d}
Then $s_1,\ldots,s_g$ are distinct elements in $\cA$.
\end{lemma}

\begin{proof}
Since $p$ and $d$ are coprime, the integers $s_i$'s are distinct and lie in $\{1,2,\ldots,d-1\}$ for $i=1,\dots,g$.
It remains to show that for $d=2g+2$ we have $s_i\ne d/2$.
Suppose there exists $1\le i\le g$ such that $\frac{d}{2}=s_i$.
Write $a=(p\bmod d)$ for the least nonnegative residue, so $a$ is coprime to $d$ hence there exists $1\le b\le d-1$ such that $ab\equiv 1\bmod d$.
Since $d$ is even, $a$ and $b$ are both odd. 
On the other hand, $ai \equiv\frac{d}{2}\pmod d$. Multiplying $b$ on both sides,
$
i\equiv b\frac{d}{2}\pmod d.
$
Since $b$ is an odd integer, we have $(b-1)\frac{d}{2}\equiv 0\bmod d$, hence 
$
i\equiv b\frac{d}{2}\equiv \frac{d}{2}\pmod d.
$
Since $1\le i\le g$, this congruence implies $i=\frac{d}{2}=\frac{2g+2}{2}=g+1$.
This contradicts that $1\le i\le g$. Therefore, $s_i\neq \frac{d}{2}$.
\end{proof}

For any integer $N$ we define a finite set of lattice points in $\R^3$:
\[
\cT(N)=\left\{(t_0,t_1,t_d)\in\mathbb Z_{\ge 0}^3:
        d t_d+t_1=N,\quad t_0+t_1+t_d=\frac{p-1}{2}\right\}.
\]
This set may be empty; for example, it is empty if $N< 0$.

%First we recall an observation below. 

\begin{lemma}\label{L:unique}
Assume $\cT(N)\neq \varnothing$. Then 
$\cT(N)$ contains a unique  element
with the maximal $t_0$-coordinate, that is,
$(t_0,t_1,t_d)=(\frac{p-1}{2}-q-r,r,q)$
where $q=\pfloor{N/d}$ and $r=[N]_d$.
\end{lemma}
\begin{proof}

Write \(N=qd+r\) with \(0\le r<d\). Every integral solution of
\(t_1+dt_d=N\) is of the form
$
t_d=q-\ell, t_1=r+d\ell
$
for some \(\ell\in\mathbb Z\). Since \(t_1\ge0\) and \(0\le r<d\),
we must have \(\ell\ge0\). Moreover,
$
t_0=\frac{p-1}{2}-q-r-\ell(d-1).
$
Thus \(t_0\) strictly decreases as \(\ell\) increases. Since
\(\cT(N)\neq\varnothing\), some feasible \(\ell\ge0\) exists; replacing
it by \(\ell=0\) increases both \(t_d\) and \(t_0\), so \(\ell=0\)
is feasible. It is therefore the unique value maximizing \(t_0\).
\end{proof}

Set $\deg(0)=-\infty$.
Define 
$\delta_{ij}=\1_{\{j>s_i\}}\coloneqq 1$ if $j>s_i$ and $0$ if $j\le s_i$.
Denote 
\begin{equation}\label{E:eta}
q_{ij}=Q_i-\delta_{ij},
\qquad
r_{ij}=s_i-j+d\delta_{ij},\qquad
\eta_{ij}=\frac{p-1}{2}-q_{ij}-r_{ij}.
\end{equation}
Applying the above lemma to $N=pi-j$, we derive the following.

\begin{lemma}\label{L:pi-j}
Let $1\le i,j\le g$. 
If $\cT(p\,i-j)\ne \varnothing$,
then it contains a unique element with maximal $t_0$-coordinate, 
that is, 
$
(t_0,t_1,t_d)=
\left(\eta_{ij},r_{ij},
q_{ij}\right).
 $
\end{lemma}
\begin{proof}
Division algorithm gives 
$$
pi-j =
\begin{cases}
Q_id+(s_i-j) & \mbox{ if $s_i\ge j$}\\
(Q_i-1)d+(d+s_i-j) & \mbox{ if $s_i<j$}.
\end{cases}
$$
The rest follows from Lemma \ref{L:unique}.
\end{proof}

For each pair $(i,j)$ with $1\le i,j \le g$, 
we define a coefficient polynomial in $\Z[X,Y]$:
\begin{equation}\label{E:Mij}
 \boxed{   M_{ij}(X,Y) \coloneqq 
    \sum_{(t_0,t_1,t_d)\in \cT(pi-j)}
    \binom{\frac{p-1}{2}}{t_0,t_1,t_d}X^{t_1}Y^{t_0},}
\end{equation}
or equivalently,
\begin{equation*}
M_{ij}(X,Y) =
[x^{pi-j}]\bigl(x^d+Xx+Y\bigr)^{(p-1)/2}.
\end{equation*}

For each pair $(i,j)$
we denote:
\begin{align}
b_{ij}&\coloneqq
\displaystyle\binom{\frac{p-1}{2}}{q_{ij}} 
\binom{\frac{p-1}{2}-q_{ij}}{r_{ij}}.
\label{E:b}
\end{align}

\begin{lemma}\label{L:entry}
If $\eta_{ij}<0$ then $M_{ij}(X,Y)=0$. 
We have $\cT(p\,i-j)\ne \varnothing$ if and only if $\eta_{ij}\ge 0$. 
If $\eta_{ij}\ge 0$, then 
$$
\deg_Y(M_{ij}(X,Y))=  \eta_{ij},\qquad
[Y^{\eta_{ij}}]M_{ij}(X,Y)=X^{r_{ij}}b_{ij}.
$$
In particular, for every algebraic integer $\alpha$, 
$$\deg_Y M_{ij}(\alpha,Y)\le \eta_{ij},
\qquad
[Y^{\eta_{ij}}]M_{ij}(\alpha,Y)=\alpha^{r_{ij}}b_{ij};
$$
if $\alpha\ne 0$, then equality holds in the degree bound.
\end{lemma}

\begin{proof}
By the definition of $q_{ij}$ and $r_{ij}$, we have
$pi-j=dq_{ij}+r_{ij}$ where $0\le r_{ij}<d$.
Suppose that $\cT(pi-j)\ne \varnothing$.
Lemma \ref{L:unique} shows that its unique element with maximal $t_0$-coordinate 
is $(\eta_{ij},r_{ij},q_{ij})$.
In particular, $\eta_{ij}\ge 0$.
Conversely, suppose that $\eta_{ij}\ge 0$. We
claim that $q_{ij}\ge 0$. If $q_{ij}<0$ then $pi-j<0$. Since $1\le j\le g<d$, this forces
$Q_i=0, s_i=pi, \delta_{ij}=1$.
Therefore
$$
\eta_{ij}=\frac{p-1}{2}+j-pi-(d-1)\le \frac{p-1}{2}+g-p-d+1=g-d-\frac{p-1}{2}<0.
$$
contradicting $\eta_{ij}\ge 0$. Thus $q_{ij}\ge 0$.
Consequently, 
$
(\eta_{ij}, r_{ij}, q_{ij})\in\Z^3_{\ge 0}
$
and satisfies the two defining equations of $\cT(pi-j)$. Hence
$\cT(pi-j)\ne \varnothing$. 

If $\eta_{ij}<0$, the equivalence just proved shows that $\cT(pi-j)=\varnothing$, 
and hence 
$M_{ij}(X,Y)=0$.
Now suppose $\eta_{ij}\ge 0$. By Lemma \ref{L:pi-j},
the unique summand of maximal $Y$-degree corresponds to 
$(t_0,t_1,t_d)=(\eta_{ij}, r_{ij}, q_{ij})$.
Its coefficient is  
$$
[Y^{\eta_{ij}}]M_{ij}(X,Y)
=X^{r_{ij}}\binom{\frac{p-1}{2}}{\eta_{ij},r_{ij},q_{ij}}
=X^{r_{ij}}b_{ij}.
$$
This is nonzero in $\Z[X]$ since $b_{ij}\in \Z_{>0}$. 
Therefore, 
$\deg_Y M_{ij}(X,Y)=\eta_{ij}$.
Specializing $X=\alpha$ gives 
$$
[Y^{\eta_{ij}}]M_{ij}(\alpha,Y)=\alpha^{r_{ij}}b_{ij},\qquad
\deg_Y(M_{ij}(\alpha,Y))\le \eta_{ij}.
$$
If $\alpha\ne 0$, the displayed coefficient is nonzero, so equality of degrees holds.
\end{proof}

\subsection{Intergral Hasse--Witt matrix}

For our study of ordinarity, the $g\times g$ coefficient matrix 
$(M_{ij}(X,Y))_{1\le i,j\le g}$, and especially its determinant, will be analyzed. 

\begin{definition}\label{D:polynomial}
Define an {\bf integral Hasse--Witt determinant polynomial} in variables $X$ and $Y$:
$$
H_{p}(X,Y):=\det_{1\le i,j\le g} M_{ij}(X,Y)\in\Z[X,Y].
$$
For every algebraic integer $\alpha$, define its specialization
$$\boxed{H_{p,\alpha}(Y):=H_p(\alpha,Y)=\det_{1\le i,j\le g} M_{ij}(\alpha,Y).}$$
When no confusion is possible, we call these {\bf Hasse--Witt polynomials}.
\end{definition}

Note that $H_p(X,Y)\in\Z[X,Y]$ and $H_{p,\alpha}(Y)$ lies in 
$\cO_{\Q(\alpha)}[Y]$. 
We will see later in Proposition \ref{P:nonzero}
that $H_{p,\alpha}(Y)$ is nonzero when $\alpha\ne 0$ and $p>P^+(d)$.

\subsection{Leading term in Hasse--Witt determinant polynomials}

The Hasse--Witt determinant 
polynomial $H_{p,\alpha}(Y)$ is a polynomial in 
$\cO_{\Q(\alpha)}[Y]$. For the proof of our main theorem, 
we need to examine its reduction mod $\fp$ for primes 
$\fp$ of  $\cO_{\Q(\alpha)}$.
There are two major steps.
We first study the leading term;
then we control the prime factors in the leading coefficient.

Let $S_g$ be the symmetric group on $g$ letters, and consider the subset 
$$
S_g':=\{\sigma\in S_g\mid \sigma(i)\le s_i \mbox{ for $1\le i\le g$}\}.
$$
For every $\sigma\in S_g$ set 
$$E(\sigma)\coloneqq\sum_{i=1}^g \eta_{i,\sigma(i)}.$$

\begin{lemma}\label{L:E}
We have $S_g'\ne \varnothing$, and $E(\sigma)$ is independent of 
$\sigma\in S_g'$. Then 
$
E=\sum_{i=1}^g \left(\frac{p-1}{2}-Q_i-s_i\right)+\frac{g(g+1)}{2}.
$
\end{lemma}
\begin{proof}
By Lemma \ref{L:S_d}, $s_1,\ldots,s_g$ are distinct positive integers;
so these integers can be arranged in increasing order $1\le s_{m_1}<s_{m_2}<\ldots<s_{m_g}$. 
Define $\sigma'(m_k)=k$ for every $k$. Since $k\le s_{m_k}$, we see that 
$\sigma'\in S_g'$. Thus
$S_g'$ is non-empty. 
By \eqref{E:eta},
$E=\sum_{i=1}^g\left(\frac{p-1}{2}-Q_i-s_i\right)+\frac{g(g+1)}2.$
This is clearly independent of $\sigma$.
\end{proof}

Define a determinant:
\begin{equation}\label{E:Delta_p}
\boxed{\Delta_p \coloneqq \det_{1\le i,j\le g}
    \binom{\frac{p-1}{2}-Q_i}{s_i-j}.}
    \end{equation}
As $p$ is odd, $\frac{p-1}{2}$ is an integer.
Since $Q_i\le \frac{p\,i}{d}<\frac{p}{2}=\frac{p-1}{2}+\frac{1}{2}$,
and $Q_i$ is integer, we have 
$Q_i\le \frac{p-1}{2}$. This implies
the binomial $\binom{\frac{p-1}{2}-Q_i}{s_i-j}\in\Z_{\ge 0}$,
hence $\Delta_p$ is an integer too. 

\begin{proposition}\label{P:D_p} %Proposition 3.7
Let
$m=\sum_{i=1}^gs_i-\frac{g(g+1)}{2}$ and $u_p=
    \prod_{i=1}^g\binom{\frac{p-1}{2}}{Q_i}$.
Then $m\ge 0$ and $u_p$ is a positive integer coprime to $p$. 
\begin{enumerate}
\item If $E\ge 0$ then 
$$\deg_Y(H_{p}(X,Y))\le E, \qquad
[Y^E]H_{p}(X,Y) =X^m u_p \Delta_p.$$
\item Let $E\ge 0$. For every commutative ring $R$ and every $a \in R$, the specialization $H_p(a,Y)\in R[Y]$ satisfies 
$$
\deg_YH_p(a,Y)\le E, \qquad 
[Y^E]H_p(a,Y)=a^m u_p \Delta_p.
$$
In particular, if $m=0$, we have $[Y^E]H_p(0,Y)=u_p\Delta_p.$
\item 
If $E<0$ then $H_{p}(X,Y)=0$ and $\Delta_p=0$.
\end{enumerate}
\end{proposition}

\begin{proof}
Since $s_1,\ldots,s_g$ are distinct positive integers, 
$
\sum_{i=1}^{g}s_i\geq \sum_{i=1}^{g}i=\frac{g(g+1)}2,
$
and hence $m\geq0$.
Set
$
n=\frac{p-1}{2}.
$
Since $1\leq i\leq g<d/2$, we have
$
Q_i=\left\lfloor\frac{pi}{d}\right\rfloor<\frac p2.
$
Because $Q_i$ is an integer, it follows that
$
0\leq Q_i\leq \frac{p-1}{2}=n.
$
Thus $\binom{n}{Q_i}$ is a positive integer. Moreover, $n<p$, so none of the factorials occurring in
$
\binom{n}{Q_i}=\frac{n!}{Q_i!(n-Q_i)!}
$
is divisible by $p$. Hence each 
$\binom{n}{Q_i}$, and therefore their product $u_p$, is coprime to $p$.

\noindent
{\bf(1).} 
By Lemma \ref{L:entry}, if $\eta_{ij}<0$, then  $M_{ij}(X,Y)=0$.
If $\eta_{ij}\ge 0$, then the unique element of $\cT(pi-j)$ with maximal
$t_0$-coordinate is 
$(t_0,t_1,t_d)=(\eta_{ij},r_{ij},q_{ij})$.
For a permutation $\sigma\in S_g$, define
$
e(\sigma):=\#\{i:\sigma(i)>s_i\}.
$
Using the definition of $\eta_{ij}$ in \eqref{E:eta} and the equality
$
\sum_{i=1}^{g}\sigma(i)=\frac{g(g+1)}2,
$
we obtain
\begin{align}\label{E:eta_sum}
\sum_{i=1}^{g}\eta_{i,\sigma(i)}
&=
\sum_{i=1}^{g}
\left(
n-q_{i,\sigma(i)}-r_{i,\sigma(i)}\right)=
E-(d-1)e(\sigma).
\end{align}
Suppose first that $E\ge 0$. Consider the Leibniz expansion
$$
H_p(X,Y)=\sum_{\sigma\in S_g}\sgn(\sigma)\prod_{i=1}^g M_{i,\sigma(i)}(X,Y).
$$
Whenever the product corresponding to $\sigma$ is nonzero, all its
factors are nonzero, and hence by \eqref{E:eta_sum},
$$
\deg_Y\prod_{i=1}^g M_{i,\sigma(i)}(X,Y)=
\sum_{i=1}^{g}\eta_{i,\sigma(i)}
=E-(d-1)e(\sigma)\leq E.
$$
Thus $\deg_YH_p(X,Y)\le E$.
Moreover, if $e(\sigma)>0$, then every nonzero product corresponding to $\sigma$ has degree $<E$. Therefore, only permutations in 
$$
S'_g=\{\sigma\in S_g: \sigma(i)\leq s_i\text{ for every $i$}\}
$$
can contribute to $[Y^E]H_p(X,Y)$.
Let $\sigma\in S_g'$. 
If $\prod_{i=1}^g M_{i,\sigma(i)}(X,Y)\ne 0$, 
then 
\begin{align}
[Y^E]\prod_{i=1}^{g} M_{i,\sigma(i)}(X,Y)
&=X^{\sum_{i=1}^g (s_i-\sigma(i))} \prod_{i=1}^gb_{i,\sigma(i)}
=X^m \prod_{i=1}^g b_{i,\sigma(i)}.
\label{E:11}
\end{align}
This formula remains valid if one of the factors 
$M_{i,\sigma(i)}(X,Y)
$
is zero. Indeed, in that case $\eta_{i,\sigma(i)}<0$ for some $i$.
Since $\sigma(i)\le s_i$,
$
\eta_{i,\sigma(i)}=n-Q_i-(s_i-\sigma(i))<0,
$
so
$
s_i-\sigma(i)>n-Q_i.
$
Therefore
$
\binom{n-Q_i}{s_i-\sigma(i)}=0,
$
and hence $b_{i,\sigma(i)}=0$, 
so both sides of \eqref{E:11} vanish.
Therefore
\begin{align}\label{E:12}
[Y^E]H_p(X,Y)
&=X^m\sum_{\sigma\in S'_g}\sgn(\sigma)\prod_{i=1}^{g}b_{i,\sigma(i)}
\\
&=X^m\left(
\prod_{i=1}^{g}\binom{n}{Q_i}
\right)
\sum_{\sigma\in S'_g}
\sgn(\sigma)
\prod_{i=1}^{g}
\binom{n-Q_i}{s_i-\sigma(i)} \nonumber\\
&=X^m u_p \sum_{\sigma\in S'_g}\sgn(\sigma)\prod_{i=1}^{g}
\binom{n-Q_i}{s_i-\sigma(i)}. \nonumber
\end{align}

If $\sigma\in S_g\setminus S'_g$, then $\sigma(i)>s_i$ for some $i$, so
$
s_i-\sigma(i)<0.
$
By our convention for binomial coefficients,
$
\binom{n-Q_i}{s_i-\sigma(i)}=0.
$
We may therefore extend the sum in \eqref{E:12} from $S'_g$ to all of $S_g$. Using the Leibniz expansion of the determinant $\Delta_p$, we obtain
$$
\begin{aligned}
\sum_{\sigma\in S'_g}
\operatorname{sgn}(\sigma)
\prod_{i=1}^{g}
\binom{n-Q_i}{s_i-\sigma(i)}
&=
\sum_{\sigma\in S_g}
\operatorname{sgn}(\sigma)
\prod_{i=1}^{g}
\binom{n-Q_i}{s_i-\sigma(i)}=
\Delta_p.
\end{aligned}
$$
Combining this with \eqref{E:12}, we have 
$$
[Y^E]H_p(X,Y)=X^m u_p\Delta_p.
$$
{\bf (2).} Now let $R$ be a commutative ring and $a\in R$.
Applying the evaluation homomorphism  
$
\Z[X,Y]\longrightarrow R[Y]$ defined by
$X\longmapsto a,$
gives the specialization $H_p(a,Y)$.
Specialization cannot increase the $Y$-degree, so
$
\deg_Y H_p(a,Y)\leq E,
$
and the coefficient identity specializes to
$$
[Y^E]H_p(a,Y)=a^m u_p\Delta_p.
$$
In particular, when $m=0$, the polynomial $X^m=X^0=1$, 
and hence 
$$[Y^E]H_p(0,Y)=u_p\Delta_p.$$

\noindent 
{\bf (3).} Suppose now that $E<0$. For every $\sigma\in S_g$, 
$$
\sum_{i=1}^{g}\eta_{i,\sigma(i)}
=E-(d-1)e(\sigma)<0.
$$
If $\prod_{i=1}^g M_{i,\sigma(i)}(X,Y)\ne 0$, then 
$\eta_{i,\sigma(i)}\geq0$ for every $i$. 
Their sum would then be nonnegative, contradicting the preceding inequality. Thus every term in the Leibniz expansion of $H_p(X,Y)$ vanishes, and hence
$
H_p(X,Y)=0.
$

It remains to prove $\Delta_p=0$.
Consider a permutation $\sigma\in S_g$.
If $\sigma\not\in S'_g$, then 
$\sigma(i)>s_i$ for some $i$, so the corresponding Leibniz term 
in $\Delta_p$ is zero.
Now suppose $\sigma\in S'_g$.
If the corresponding Leibniz term were nonzero, then for every $i$, 
$\binom{n-Q_i}{s_i-\sigma(i)}\ne 0$.
Since $s_i-\sigma(i)\ge 0$, this implies 
$s_i-\sigma(i)\le n-Q_i,$
and therefore
$\eta_{i,\sigma(i)}=n-Q_i-s_i+\sigma(i)\ge 0.$
Since $\sigma\in S'_g$, Lemma \ref{L:E} gives 
$
E=\sum_{i=1}^{g}\eta_{i,\sigma(i)}\geq0,
$
contrary to the assumption $E<0$. Thus every term in the Leibniz
expansion of $\Delta_p$ is zero, and therefore
$
\Delta_p=0.
$
This proves part (3).
\end{proof}

\subsection{Hasse--Witt determinant polynomial mod $p$}
  
For any prime $\fp\mid p$ in the ring of integers $\cO_{\Q(\alpha)}$, 
we write $\kappa(\fp)=\cO_{\Q(\alpha)}/\fp$.

\begin{proposition}
\label{P:nonzero}
Let $\alpha$ be a nonzero algebraic integer, and let $F=\Q(\alpha)$.
Let $p>P^+(d)$.
\begin{enumerate}
\item Then $H_{p,\alpha}(Y)$ is nonzero in $\cO_{\Q(\alpha)}[Y]$.
\item For any prime $\fp\mid p$ in $\cO_F$, if 
$\alpha\not\in\fp$ then $H_{p,\alpha}(Y) \bmod \fp$
is nonzero in $\kappa(\fp)[Y]$.
\end{enumerate}
\end{proposition}

\begin{proof}
Since \(p>P^+(d)>d\), both
\(2\) and \(d\) are invertible modulo \(p\); thus for each $i$ we have 
the congruence in $\Z_{(p)}$:
\[
\frac{p-1}{2}-Q_i
=\frac{p-1}{2}-\frac{p\,i-s_i}{d}
\equiv \frac{s_i}{d}-\frac12
=\beta_i
\pmod p.
\]
If \(0\le s_i-j<d<p\), then \((s_i-j)!\) is invertible modulo \(p\),
and therefore in $\Z_{(p)}$ we have
\[
\binom{\frac{p-1}{2}-Q_i}{s_i-j}
\equiv
\binom{\beta_i}{s_i-j}
\pmod p.
\]
If \(s_i-j<0\), both corresponding binomial coefficients are zero, so the above congruence again holds.
Taking determinants gives
\[
\Delta_p\equiv
\Delta(s_1,\ldots,s_g)\not\equiv0\pmod p
\]
by Proposition~\ref{P:determinant} (since $p>P^+(d)$). 
In particular, $\Delta_p\ne 0$ in $\Z$.
Proposition~\ref{P:D_p}(3) therefore implies that $E\ge 0$, and 
Proposition~\ref{P:D_p}(2) gives 
$[Y^E]H_{p,\alpha}=\alpha^m u_p\Delta_p.$
Since $\alpha\ne 0$, $u_p\in\Z_{\ge 1}$, and $\Delta_p\ne 0$, 
this coefficient is nonzero in $\cO_F$. This proves part (1).

Now let $\fp\mid p$ and $\alpha\not\in \fp$. 
Then $\alpha\not\equiv 0\bmod \fp$.
By Proposition \ref{P:D_p}, $u_p\not\equiv 0\bmod p$.
We also have $\Delta_p\not\equiv 0\bmod \fp$ above.
Thus $[Y^E]H_{p,\alpha}= \alpha^m u_p \Delta_p\not\equiv 0\bmod \fp$, hence 
$H_{p,\alpha}(Y)\not\equiv 0\bmod \fp$, proving part (2).
\end{proof}

\begin{corollary}\label{C:nonzero}
Let $p>P^+(d)$.  
Then 
$
H_p(X,Y) \bmod p
$
is a nonzero polynomial in $\F_p[X,Y]$. 
In particular, $H_p(X,Y)$ is a nonzero polynomial over $\Z$.
\end{corollary}
\begin{proof}
If $H_p(X,Y)\bmod p$ is a zero polynomial then 
$H_{p,1}(Y)=H_p(1,Y)=0\bmod p$, 
contradicting Proposition \ref{P:nonzero}.
\end{proof}

\section{Hyperelliptic families 
with generically ordinary fibers}

Let $\alpha$ be an algebraic integer and let  $F=\Q(\alpha)$.
The discriminant of $f(x)=x^d+\alpha x+t$ is
\begin{equation}\label{E:Disc}
D_\alpha(t)=\disc_x(f(x))=
(-1)^{d(d-1)/2}(d^dt^{d-1}+(-1)^{d-1}(d-1)^{d-1}\alpha^d).
\end{equation}
Let $p\nmid 2d$ below.
For each prime $\fp\mid p$ in $\cO_F$ with  $\kappa(\fp)=\cO_F/\fp$, let
$$
\cS_\fp^\circ \coloneqq \cS^\circ \times_{\Spec{\cO_F}} 
\Spec\; \kappa(\fp). 
$$
Let $\bar\alpha$ denote $\alpha\bmod \fp$ and 
$\bar{D}_{\bar\alpha}(t)$ denote 
$D_\alpha(t)\bmod \fp$.
Since $\bar{D}_{\bar\alpha}(t)$ is nonzero,
every fiber $\cS^\circ_\fp$ is a nonempty geometrically integral open 
subscheme of $\A^1_{\kappa(\fp)}$, defined by the principal open 
$D(\bar{D}_{\bar{\alpha}}(t))$.

\begin{theorem}[Theorem \ref{T:main}]
\label{T:main2}
Let $d\in\{2g+1,2g+2\}$ with $g\ge 2$.
Let $\alpha$ be a nonzero algebraic integer. 
For every prime $p>P^+(d)$ and $p\nmid N_{F/\Q}(\alpha)$, 
the family of genus-$g$
curves 
\begin{eqnarray*}
\cC_\alpha &\longrightarrow & \cS^\circ
\end{eqnarray*}
is generically ordinary at each prime $\fp\mid p$ in $\cO_F$. 
\end{theorem}

\begin{proof}
Each reduction $\cC_\alpha\bmod \fp$ has affine equation $y^2=\bar{f}(x)$
where 
$\bar{f}(x)=x^d+\bar{\alpha}x+t
$ lies in $\kappa(\fp)[t][x]$.
The generic fiber over $\kappa(\fp)(t)$ is ordinary if 
and only if 
$$\det_{1\le i,j\le g} (C_{ij})\ne 0,\qquad  
C_{ij}=[x^{p\,i-j}]\bar{f}(x)^{\frac{p-1}{2}}.
$$
(We remark that the coefficient matrix $(C_{ij})_{1\le i,j\le g}$ 
is the transpose of Hasse--Witt matrix.
See \cite{AH19} for clarification in modern notations.)
From  \eqref{E:Mij}, we immediately conclude
$$
C_{ij}= (M_{ij}(\alpha,t) \bmod \fp).
$$
Taking determinant, we have
$$\det_{1\le i,j\le g} (C_{ij})= \det_{1\le i,j\le g} (M_{ij}(\alpha,t)\bmod \fp)
=(H_{p,\alpha}(t)\bmod \fp)$$
where $H_{p,\alpha}(t)$ is the Hasse--Witt polynomial 
introduced above in Definition \ref{D:polynomial}.
Write $\bar{H}_{p,\bar\alpha}(Y)=H_{p,\alpha}(Y)\bmod \fp$.
Then the generic fiber over $\kappa(\fp)(t)$ 
is ordinary if and only if 
$\bar{H}_{p,\bar\alpha}(t)$ is nonzero in $\kappa(\fp)[t]$.
Equivalently, the ordinary locus in $\cS_\fp^\circ$ is
$D(\bar{H}_{p,\bar\alpha})\cap \cS^\circ_{\fp}$.

Note that $p\nmid N_{F/\Q}(\alpha)$ implies $\alpha\not\in\fp$ for every 
prime $\fp\mid p$. Applying 
Proposition~\ref{P:nonzero}~(2), 
we see that $\bar{H}_{p,\bar\alpha}(t)$ 
is a nonzero polynomial in $\kappa(\fp)[t]$ for every $\fp\mid p$. 
Hence this ordinary locus on $\cS_\fp^\circ$ is  $D(\bar{H}_{p,\bar\alpha})\cap \cS_\fp^\circ$.
As noted before the theorem, 
the ordinary locus on $\A^1_{\kappa(\fp)}$ 
is $D(\bar{D}_{\bar\alpha} \bar{H}_{p,\bar\alpha})$.
This proves that $\cC_\alpha$ is generically ordinary.
\end{proof}

\begin{corollary}\label{C:integer}
Let $d\in\{2g+1,2g+2\}$ with $g\ge 2$.
Let $\alpha$ be a nonzero integer.
For every prime $p>P^+(d)$ and $p\nmid \alpha$,
the family of genus-$g$ curves 
\begin{eqnarray*}
\cC_\alpha &\longrightarrow & \cS^\circ
\end{eqnarray*}
is generically ordinary at $p$. In particular, the Katz family 
$\cK$ is generically ordinary for all $p>P^+(d)$.
\end{corollary}

\begin{proof}
For a nonzero integer $\alpha$, apply Theorem~\ref{T:main2} with
$F=\Q$. The hypothesis is precisely $p\nmid\alpha$. This proves the first
assertion.
For the Katz family, take $\alpha=-d$. Since $p>P^+(d)>d$, one has 
$p\nmid d$, and the result follows.
\end{proof}

Consider the family $\cC$ of curves 
given by affine chart $y^2=f(x)$ where
$$f(x)=x^d+sx+t$$ over the $(s,t)$--plane. 
Write $D(s,t)=\disc_x(f(x))$.
According to \cite{Kat14}, the topological monodromy
of the two-parameter family $\cC$ has finite index in $\Sp_{2g}(\Z)$.

\begin{corollary}\label{C:plane}
Let $d\in\{2g+1,2g+2\}$ with $g\ge 2$.
Then the family $\cC$ of genus-$g$ curves 
\begin{eqnarray*}
\cC&\longrightarrow & \Spec\Z[1/2,s,t,D(s,t)^{-1}]
\end{eqnarray*}
is generically ordinary at each prime $p>P^+(d)$. 
\end{corollary}
\begin{proof}
For $p>P^+(d)$, the reduction $\bar{D}(s,t)$ at $p$ 
is nonzero.
The Hasse--Witt polynomial
of the generic fiber over $\F_p(s,t)$ is 
$\bar{H}_p(s,t)$. 
Corollary \ref{C:nonzero} shows this is nonzero in $\F_p[s,t]$.
Since $\bar{D}(s,t)\bar{H}_p(s,t)$ 
is a nonzero polynomial in $\F_p[s,t]$, the corresponding 
nonvanishing locus is a nonempty open subset of the base.
Therefore the generic fiber is ordinary.
\end{proof}

\begin{remark}
Each curve $\cC_\alpha(t)$ is isomorphic to 
$\cK(t)$ after some finite extension of base field and scaling 
of parameters. More precisely, 
for any nonzero $\alpha, \beta$, choose $u$ such that $u^{2(d-1)}=\alpha/\beta$.
Then $x=u^2X, y=u^dY, t=u^{2d}T$ transforms 
$\cC_\alpha(t): y^2=x^d+\alpha x+t$ to $\cC_\beta(T): Y^2=X^d+\beta X+T$.
Thus the families become isomorphic after a finite extension of the ground field and a linear rescaling of the parameter.
\end{remark}

For the residue class $p\equiv 1\pmod d$, the bound $P^+(d)$ can be improved.

\begin{corollary}[Corollary \ref{C:bound}]\label{C:bound-proof}
Let $\alpha$ be any algebraic integer, including the case $\alpha=0$, and let $F=\Q(\alpha)$.
Then for every prime $p\equiv 1\bmod d$ 
$$\cC_\alpha \longrightarrow \cS^\circ$$
is generically ordinary at each prime $\fp\mid p$ of $F$.
\end{corollary}

\begin{proof}
Since $p\equiv 1\pmod d$, one has $p>d$, and therefore $p\nmid d$.
The coefficient of $t^{d-1}$ in $D_\alpha(t)$ remains nonzero modulo $\fp$, as seen in 
\eqref{E:Disc}, $\bar{D}_\alpha(t)\ne 0$.
We have $s_i=i$ for $1\le i\le g$.
Hence the matrix defining $\Delta_p$ is lower triangular with diagonal
entries $1$, so $\Delta_p=1$. 
In particular, $\Delta_p\neq 0$, so
Proposition~\ref{P:D_p}(3) implies $E\ge 0$. 
Moreover, 
$m=\sum_{i=1}^g s_i-\frac{g(g+1)}{2}=0$.
By Proposition \ref{P:D_p}(1),
$$
[Y^E]H_{p}(X,Y)=X^mu_p\Delta_p=u_p. 
$$
Let $\fp|p$, and let $\bar\alpha\in \kappa(\fp)$ denote
the reduction of $\alpha$. By Proposition \ref{P:D_p}(2),
$$
[Y^E]H_{p,\bar{\alpha}}(Y)=\bar{u_p}\ne 0$$
since $p\nmid u_p$.
Thus $H_{p,\bar\alpha}(Y)$ is nonzero for every $\bar\alpha$,
including $\bar\alpha=0$.
The argument of Theorem \ref{T:main2} 
now shows that $\cC_\alpha$ is generically ordinary at $\fp$.
\end{proof}

\end{document}